\documentclass[pdftex, oneside, 10pt, letterpaper]{amsart}

\usepackage{verbatim} 
\usepackage{amsthm, amssymb, amsmath}

\usepackage{graphicx}
\usepackage[usenames, dvipsnames]{color}
\usepackage{tikz}
\usetikzlibrary{plotmarks}
\usepackage{pgfplots}
\pgfplotsset{compat=newest}
\newlength\figurewidth

\usepackage{listings}
\usepackage{upquote} 

\usepackage[activate={true,nocompatibility},spacing,kerning]{microtype}

\input{glyphtounicode} \pdfgentounicode=1

\usepackage{geometry}

\usepackage[pdftex,%
            hyperfigures,%
            hidelinks,%
            pdftitle={The Singular Values of the GOE},%
            pdfdisplaydoctitle=true,%
            breaklinks=false,%
            bookmarks=true,%
            unicode=true,
            pdffitwindow=true,%
            pdfauthor={Folkmar Bornemann and Peter Forrester},%
            pdfsubject={Random Matrices},%
            pdfkeywords={random matrices, evenness symmetry,
              values, gap probabilities},%
]{hyperref} 

\newsavebox{\gwbox}\newlength{\gwidth}

\newcommand{\defdelim}[1]{\ifx\relax#1\relax\def\lsize{\left}\def\rsize{\right}\else\def\lsize{#1}\def\rsize{#1}\fi}
\newcommand{\floor}[2][]{\defdelim{#1}\lsize\lfloor#2\rsize\rfloor}
\newcommand{\ceil}[2][]{\defdelim{#1}\lsize\lceil#2\rsize\rceil}

\def\<#1>{\left\langle{#1}\right\rangle}

\theoremstyle{plain}
\newtheorem{thm}{Theorem}
\newtheorem{lemma}{Lemma}
\newtheorem{cor}{Corollary}

\theoremstyle{definition}

\theoremstyle{remark}
\newtheorem*{rem}{Remark}

\renewcommand{\leq}{\leqslant}
\renewcommand{\geq}{\geqslant}
\renewcommand{\le}{\leqslant}

\renewcommand{\omega}{\varpi}

\DeclareMathOperator{\diag}{\ensuremath{\mathrm{diag}}}
\DeclareMathOperator{\sign}{\ensuremath{\mathrm{sgn}}}
\DeclareMathOperator{\erf}{\ensuremath{\mathrm{erf}}}
\DeclareMathOperator{\even}{\ensuremath{\mathrm{even}}}
\DeclareMathOperator{\odd}{\ensuremath{\mathrm{odd}}}

\newcommand{\chUE}{\ensuremath{\mathrm{chUE}}}
\newcommand{\UE}{\ensuremath{\mathrm{UE}}}
\renewcommand{\OE}{\ensuremath{\mathrm{OE}}}

\newcommand{\R}{\ensuremath{\mathbb{R}}}

\title[Singular Values and Evenness Symmetry in Random Matrix Theory]{Singular Values and Evenness Symmetry in Random Matrix Theory}

\author{Folkmar Bornemann}
\address[Folkmar Bornemann]{Zentrum Mathematik -- M3, 
  Technische Universit\"at M\"unchen, Germany}
\email{bornemann@tum.de}

\author{Peter J. Forrester}

\address[Peter J. Forrester]{Department of Mathematics and Statistics, University of Melbourne, Australia; ARC Centre of Excellence for Mathematical \& Statistical Frontiers}
\email{p.forrester@ms.unimelb.edu.au}

\keywords{random matrices, evenness symmetry, singular values, gap probabilities}

\begin{document}
\begin{abstract} 
Complex Hermitian random matrices with a unitary symmetry can be distinguished by a weight function. When this is even,
it is a known result that the distribution of the singular values can be decomposed as the superposition of two independent eigenvalue sequences
distributed according to particular matrix ensembles with chiral unitary symmetry. We give decompositions of the distribution of singular values,
and the decimation of the singular values --- whereby only even, or odd, labels are observed --- for real symmetric random matrices
with an orthogonal symmetry, and even weight. This requires further specifying the functional form of the weight to one
of three types --- Gauss, symmetric Jacobi or Cauchy. Inter-relations between gap probabilities with orthogonal and unitary symmetry follow as
a corollary. The Gauss case has appeared in a recent work of Bornemann and La Croix. The Cauchy case, when appropriately specialised and upon
stereographic projection, gives decompositions for the analogue of the singular values for the circular unitary and circular orthogonal ensembles.
\end{abstract}

\maketitle


\section{Introduction}\label{section:introduction}

The ensembles of real symmetric random matrices $\OE_n(w_1)$ possessing an orthogonal symmetry, and complex Hermitian random matrices
$\UE_n(w_2)$ possessing a unitary symmetry, are specified by the eigenvalue densities
\begin{equation}\label{eqn:goedensity}
p_\beta(x_1,\ldots,x_n) = c_{n,\beta} \prod_{k=1}^n w_\beta(x_k)\cdot |\Delta(x_1,\ldots,x_n)|^\beta\qquad (\beta=1,2)
\end{equation}
with some normalization constant $c_{n,\beta}$, each $x_k$ restricted to the interval of support of $w_\beta(x_k)$, and the Vandermonde determinant\footnote{Note that $\Delta(\xi_1,\ldots,\xi_n)\geq0$ if the arguments are increasingly ordered, $\xi_1\leq\cdots\leq \xi_n$.}
\[
\Delta(\xi_1,\ldots, \xi_n) = \det\begin{pmatrix}
1 & 1 & \cdots & 1 \\
\xi_1 & \xi_2 & \cdots & \xi_n\\
\vdots & \vdots & & \vdots\\
\xi_1^{n-1} & \xi_2^{n-1}& \cdots & \xi_n^{n-1} 
\end{pmatrix} = \prod_{k>j} (\xi_k-\xi_j).
\]
Further, relating to a chiral unitary symmetry, there is the matrix ensemble $\chUE(w_2)$ with {\em positive} eigenvalues
distributed according to the density, see \cite[p.~717]{EvenSymm},
\begin{equation}\label{eqn:chiraldensity}
p_{\text{ch}}(x_1,\ldots,x_n) = c_{n}^\text{ch} \prod_{k=1}^n w_2(x_k)\cdot \Delta(x_1^2,\ldots,x_n^2)^2.
\end{equation}
As in the theory of orthogonal polynomials, the $w_\beta(x)$ are referred to as weights. In fact 
the ensembles are  often referred to by the name for the weights  used in the theory of orthogonal polynomials.
For example, OE${}_N(e^{-x^2/2})$ is referred to as the Gaussian orthogonal ensemble.

In this paper, as a unifying framework for examining eigenvalue properties under evenness symmetry, 
introduced into random matrix theory in the works \cite{Ra03,EvenSymm} and further explored in the Gaussian case
in the recent works \cite{ELC14, BLC},
we study the structure of the singular values of ensembles $\OE_n(w_1)$ 
 with \emph{even} weights $w_1$ supported
on $(-\omega,\omega)$ as given in Table~\ref{tab:admissible1}. 
\begin{table}[htdp]
\caption{admissible pairs of symmetric weights supported on $(-\omega,\omega)$; $a>-1$}
{\begin{center}
\begin{tabular}{cccc}
case & $w_1(x)$ & $w_2(x)$ & $\omega$\\*[1mm]\hline
Gauss$\phantom{\Big|}$  & $e^{-x^2/2}$& $e^{-x^2}$ & $\infty$\\
Jacobi$\phantom{\Big|}$ & $(1-x^2)^a$  & $(1-x^2)^{2a+1}$ & $1$\\
Cauchy$\phantom{\Big|}$ & $(1+x^2)^{-(n+a+1)/2}$ & $(1+x^2)^{-(n+a)}$ & $\infty$\\*[1mm]\hline
\end{tabular}
\end{center}}
\label{tab:admissible1}
\end{table}%
The ensemble of singular values will be briefly denoted
by $|\OE_n(w_1)|$, in keeping with the relationship between the eigenvalues and singular values --- since the
ensembles are Hermitian, the singular values are the absolute value of the eigenvalues.
Although defined according to the probability density function (\ref{eqn:goedensity}), we remark that each ensemble implied by
Table \ref{tab:admissible1} can be realised in terms of matrix ensembles defined by a distribution on the elements (see e.g.~\cite[Ch.~1--3]{Fo10}).

Central to our discussion is  the operation of {\em decimation}, which if applied to $|\OE_n(w_1)|$ results in the two
ensembles
 \[
\even |\OE_n(w_1)| \quad\;\text{and}\quad\; \odd |\OE_n(w_1)|,
\]
where we define the even-location decimated ensemble $\even |\OE_n(w_1)|$
by taking the 2nd largest, 4th largest etc.  singular value, and
similarly for $\odd |\OE_n(w_1)|$. The results will often depend on the parity~$\mu$ of
the underlying order $n$ and we will, throughout this paper, write
\begin{subequations}\label{eqn:dim}
\begin{equation}
  n = 2m + \mu \quad (\mu=0,1),\qquad \hat{m}=m+\mu,
\end{equation}
that is,  
\begin{equation}
  m = \floor{n/2}, \qquad \hat{m} = \ceil{n/2}, 
  \qquad \mu = \ceil{n/2} - \floor{n/2}.
\end{equation}
\end{subequations}
Then, generalizing the corresponding result of Bornemann and La Croix \cite[Thm.~1]{BLC} for Gaussian ensembles,
the following structure holds.

\begin{thm}\label{thm:main1} Let $w_\beta$ ($\beta=1,2$) be the weight pairs of the Gauss, symmetric Jacobi or Cauchy case as given in
Table~\ref{tab:admissible1}. Denoting equality of the joint
  distribution of two ensembles by $\overset{\rm d}{=}$, there holds
\begin{equation}\label{eqn:even_aGUE}
\even |\OE_n(w_1)| \,\overset{\rm d}{=}\, \chUE_m(x^{2\mu} w_2)\qquad (n=2m+\mu).
\end{equation}
\end{thm}

If we recall the
superposition representation \cite[Eq.~(2.6)]{EvenSymm}
\begin{equation}\label{1.17}
|\UE_n(w_2)| \,\overset{\rm d}{=}\,  \chUE_{\hat{m}}(w_2)\,\cup\; \chUE_{m}(x^2 w_2),
\end{equation}
of the singular values of the corresponding
unitary ensemble $\UE_n(w_2)$,
with both ensembles on the right drawn independently,
Theorem~\ref{thm:main1} immediately implies the following remarkable
relation between the singular values of $\OE(w_1)$ and $\UE(w_2)$:
\begin{cor}\label{cor:main1} Let $w_\beta$ ($\beta=1,2$) be the weight pairs of the Gauss, symmetric Jacobi or Cauchy case as given in
Table~\ref{tab:admissible1}. Then, with the ensembles on the right drawn independently, there holds
\begin{equation}\label{eqn:super}
|\UE_n(w_2)|  \,\overset{\rm d}{=}\,  \even |\OE_n(w_1)| \,\cup\, \even |\OE_{n+1}(w_1)|.
\end{equation}
\end{cor}

\noindent
The superposition \eqref{eqn:super} bears a striking
similarity with a corresponding superposition result for the eigenvalue
distributions, see \cite[pp.~185--186]{FR} or \cite[\S6.6]{Fo10}, namely
\[
\UE_n(w_2) \,\overset{\rm d}{=}\,  \even\left( \OE_n(w_1) \cup\, \OE_{n+1}(w_1)\right).
\]

We proceed as follows: first, in Section~\ref{section:revision} we give an
overview of superposition and decimation results in random matrix theory known from previous
studies, so as to properly set the scene for the present study and also as an opportunity 
to introduce the circular ensembles. In Section~\ref{section:jointdensity} 
a  factorized expression for
the joint density of the singular values is obtained, with the proof of Theorem~\ref{thm:main1} in Section~\ref{section:evensingularvalues}
 following from this by integrating out the
odd-location singular values. The success of this task is based on the notion of \emph{admissible} symmetric weights, which we
introduce in Section~\ref{section:admissible}. There, Theorem~\ref{thm:classification} will give a complete
classification of all the admissible weights, namely, they are exactly the Gauss, symmetric Jacobi and Cauchy weights
(this is not to say that Theorem~\ref{thm:main1} would not hold for other ensembles, but to point out that the method of proof is limited
to those cases). In the first subsection of
Section~\ref{section:gapprobabilities}, some inter-relationships between gap probabilities are deduced from Theorem~\ref{thm:main1}.
For $n$ even, these have been obtained in the earlier study \cite{EvenSymm} without knowledge of Theorem~\ref{thm:main1}. We proceed to provide the necessary
working to show that this is still possible for $n$ odd. The relative complexity serves to further highlight the advantages of a viewpoint based on singular values.
We conclude in Section~\ref{section:circular} by presenting a number of new inter-relations between the spectra of circular ensembles, which follow upon the
use a stereographic projection of the appropriate Cauchy weights to specify
circular ensemble analogues of Theorem~\ref{thm:main1} and its various corollaries.


\section{inter-relations known from previous studies}\label{section:revision}

\subsection{Circular ensembles}

Central to our theme is the operation of superposition,
whereby eigenvalue sequences from two independent ensembles with orthogonal symmetry are superimposed, 
and that of decimation, meaning in the present context that only
those 
eigenvalues with a particular parity in the ordering are observed. The best known example
of these operations involves not eigenvalues on the real line as in (\ref{eqn:goedensity}), but rather matrix
ensembles with all eigenvalues on the unit circle in the complex plane. In fact such ensembles naturally
follow from (\ref{eqn:goedensity}) with the Cauchy weight
\begin{equation}\label{W1}
w_\beta(x) = {1 \over (1 + x^2)^{\beta (n - 1)/2 + 1}}.
\end{equation}
Thus, after making, for each eigenvalue, the change of variables 
\begin{equation}\label{W1a}
e^{i \theta} = {1 + i x \over 1 - i x},\qquad x = \tan(\theta/2),
\end{equation}
corresponding to a stereographic mapping from the real line to the unit circle, one obtains the eigenvalue PDF on
the unit circle
\begin{equation}\label{1A}
\propto |\Delta(e^{i\theta_1},\ldots,e^{i\theta_n})|^\beta,
\end{equation}
referred to, in the case $\beta=1$, as the circular orthogonal ensemble ${\rm COE}_n$ and, in the case $\beta = 2$, 
as the circular unitary ensemble ${\rm CUE}_n$; see e.g.~\cite[Ch.~2]{Fo10}.

Let us superimpose two independent COE${}_n$ ensembles to obtain a new sequence of eigen-angles
$$
0 < \theta_1 < \theta_2 < \cdots < \theta_{2n} < 2 \pi,
$$
and denote it by COE${}_n\, \cup \,$COE${}_n$. It was conjectured by Dyson \cite{Dy62} and proved by Gunson \cite{Gu62} that
\begin{equation}\label{21}
{\rm alt} \, ({\rm COE}_n \, \cup \, {\rm COE}_n)\,\overset{\rm d}{=}\, {\rm CUE}_n,
\end{equation}
where the alt operation refers to the integration over alternate angles $\theta_1,\theta_3,\dots,\theta_{2n-1}$ in the region
$$
\theta_{2j} < \theta_{2j+1} < \theta_{2j+2} \qquad (j=0,\dots,n-1)
$$
with $\theta_0 = \theta_{2n} - 2 \pi$. 

The inter-relation (\ref{21}) between eigenvalue distributions implies an inter-relation between
conditioned gap probabilities. These are the probabilities, denoted by $E_{n,\beta}(k;J;w_\beta)$, or alternatively by
$E_{n,\beta}(k;J; {\rm ME}{}_{n,\beta}(w_\beta))$,  that the matrix ensemble
ME${}_{n,\beta}(w_\beta)$ contains exactly $k$ eigenvalues in the interval $J$.  Then as a direct combinatorial consequence of
(\ref{21})  one has \cite{Dy62,Me92} (cf. also~(\ref{24}), (\ref{24cp}))
\begin{multline}
E_{n,2}^{}(k;(-\theta,\theta);{\rm CUE}_n)  \\
  =   \sum_{j=0}^n \Big ( E_{n,1}(2(k-j);(-\theta,\theta);{\rm COE}_n) +
E_{n,1}^{}(2(k-j)-1;(-\theta,\theta);{\rm COE}_n) \Big )\\
 \times
\Big ( E_{n,1}^{}(2j;(-\theta,\theta);{\rm COE}_n) +
E_{n,1}^{}(2j+1;(-\theta,\theta);{\rm COE}_n) \Big ). \label{8.31p}
\end{multline}

Closely related to the determinantal structure underlying the eigenvalue PDF (\ref{1A}) for the ${\rm CUE}_n$,
together with the fact that this eigenvalue PDF is unchanged by complex conjugation, is the inter-relation \cite{Ra03}
\begin{equation}\label{CO}
|{\rm CUE}_n|  \,\overset{\rm d}{=}\,   O^+(n+1) \, \cup \,  O^-(n+1).
\end{equation}
As the name suggests, here $O^\pm(n+1)$ refers to the eigen-angles of matrices from the classical groups of the same name,
chosen with Haar measure. Eigen-angles $0$ and $\pi$, which appear for purely algebraic reasons, are ignored and, since orthogonal matrices have real
entries, for each eigen-angle $\theta \ne 0, \pi$, there is another eigen-angle $-\theta$, so that we take the one within
the range $0<\theta<\pi$ only. 
 The notation $| \cdot |$ now refers to the distribution of eigen-angles in the range $0 < \theta < \pi$,
union the negative of the eigen-angles in the range $-\pi < \theta < 0$. Though $|\cdot|$ has no effect on $O^\pm(n+1)$,
this is not the case for the ${\rm CUE}_n$, where the eigenvalue distribution, and the distribution implied by
$|{\rm CUE}_n| $ are very different.

As shown in \cite[Eq.~(2.6)]{EvenSymm}, the analogue of (\ref{CO}) for Hermitian matrix ensembles with unitary symmetry is  (\ref{1.17}).
In fact (\ref{CO}) can be deduced from (\ref{1.17}) with the Cauchy weight $w_2(x) = (1+x^2)^{-n}$,  upon applying
the change of variables (\ref{W1a}) corresponding to a stereographic projection. On the RHS this requires the facts that
under the change of variable $x = \tan \theta/2$ for each eigenvalue (see \cite[Eqs.~(2.24)--(2.28)]{EvenSymm},
\begin{equation}\label{UO}
{\rm chUE}_{\hat{m}}((1+x^2)^{-n})  \,\overset{\rm d}{=}\,
O^+(n+1), \quad
{\rm chUE}_{m}(x^2(1+x^2)^{-n})\,\overset{\rm d}{=}\,
O^-(n+1),
\end{equation}
on the LHS this change of variables simply gives 
\[
|\UE_n((1+x^2)^{-n})| \,\overset{\rm d}{=}\, |{\rm CUE}_n|.
\]

\subsection{Hermitian ensembles}\label{subsect:Hermitian}

Forrester and Rains \cite{FR} considered analogues of (\ref{21}) for ensembles of Hermitian matrices. In keeping with above notations, let
OE${}_n(w_1) \, \cup \, $OE${}_n(w_1)$ denote the superimposing of two sequences of eigenvalues, independently drawn from OE${}_n(w_1)$.
Suppose the resulting eigenvalues are ordered $x_1 > x_2 > \cdots > x_{2n}$,
and let even$ \, (\OE_n(w_1) \cup \OE_n(w_1)$) refer to the distribution of the even-location eigenvalues.
We know from \cite[pp.~186--187]{FR}, see also \cite[\S6.6]{Fo10} that this is identically distributed to an ensemble with unitary symmetry,
\begin{equation}\label{24a}
{\rm even} \, ( \OE_n(w_1) \cup \OE_n(w_1) ) \,\overset{\rm d}{=}\, \UE_n(w_2),
\end{equation}
for the pairs $(w_1,w_2)$ of weights given in Table~\ref{tab:24b}
and furthermore, up to a linear fractional transformation, these pairs of weights are \emph{unique}. 
\begin{table}[htbp]
\caption{pairs of weights satisfying Eq.~\eqref{24a}; $a>-1$}
{\begin{center}
\begin{tabular}{cccc}
case & $w_1(x)$ & $w_2(x)$ & support\\*[1mm]\hline
Laguerre$\phantom{\Big|}$  & $e^{-x/2}$& $e^{-x}$ & $(0,\infty)$\\
Jacobi$\phantom{\Big|}$ & $(1 - x)^{(a-1)/2}$  & $(1-x)^a$ & $(0,1)$\\
\end{tabular}
\end{center}}
\label{tab:24b}
\end{table}%
The inter-relation between ensembles (\ref{24a}) has as an immediate combinatorial consequence the inter-relation between
gap probabilities
\begin{equation}\label{24}
E_{n,2}(k;(0,s);w_2) = \sum_{j=0}^{2k} E_{n,1} (2k-j;(0,s);w_1) \Big ( E_{n,1}(j;(0,s);w_1) +
 E_{n,1}(j-1;(0,s);w_1)  \Big ).
  \end{equation}

It is also fruitful to consider the superimposed and decimated ensemble
even (${\rm OE}{}_n(f) \, \cup \, {\rm OE}{}_{n+1}(f)$), thus involving one ensemble
with $n$ eigenvalues and the other with $n+1$. It is shown in \cite[pp.~185--186]{FR}, see also \cite[\S6.6]{Fo10}, that this, again, is identically distributed to an ensemble with unitary symmetry
\begin{equation}\label{24c}
{\rm even} \, ({\rm OE}_n(w_1) \cup {\rm OE}_{n+1}(w_1) ) \,\overset{\rm d}{=}\, {\rm UE}_n(w_2),
\end{equation}
where $(w_1,w_2)$ is any one of the pairs $(w_1,w_2)$ of weights given in Table~\ref{tab:24c} (note that Table~\ref{tab:admissible1} gives the subset of {\em even} weights).
\begin{table}[htbp]
\caption{pairs of weights satisfying Eq.~\eqref{24c}; $a,b>-1$}
{\begin{center}
\begin{tabular}{cccc}
case & $w_1(x)$ & $w_2(x)$ & support\\*[1mm]\hline
Gauss$\phantom{\Big|}$  & $e^{-x^2/2}$& $e^{-x^2}$ & $(-\infty,\infty)$\\
Laguerre$\phantom{\Big|}$  & $x^{(a-1)/2} e^{-x/2}$ & $x^a e^{-x}$ & $(0,\infty)$\\
Jacobi$\phantom{\Big|}$ & $(1+x)^{(a-1)/2} (1 - x)^{(b-1)/2}$  & $(1+x)^{a} (1 - x)^{b}$ & $(-1,1)$\\
Cauchy$\phantom{\Big|}$ & $(1 + x^2)^{-(n+a+1)/2}$ & $(1 + x^2)^{-(n+a)}$ & $(-\infty,\infty)$
\end{tabular}
\end{center}}
\label{tab:24c}
\end{table}%
As for Table~\ref{tab:24b} in relation to (\ref{24a}), these pairs of weights were shown to be \emph{unique} up to linear transformation. 
An immediate combinatorial consequence for gap probabilities is the inter-relation
\begin{equation}\label{24cp}
E_{n,2}(k;J_s;w_2) = \sum_{j=0}^{2k+1} E_{n,1} (2k+1-j;J_s;w_1) \Big ( E_{n+1,1}(j; J_s;w_1) +
 E_{n+1,1}(j-1; J_s;w_1)  \Big ),
 \end{equation}
where $J_s$ is a single interval either starting at the left boundary of support and finishing at $s$,
or starting at $s$ and finishing at the right boundary of support.

\begin{rem}
Although it has no direct bearing on the present study, there is a decimation relation relating OE${}_n(w_1)$ for the weights
in Table~\ref{tab:24c} to a corresponding PDF (\ref{eqn:goedensity}) with $\beta = 4$ \cite{MD63,FR}, which further generalises to a decimation relation
reducing ensembles with $\beta = 2/(r+1)$, $r \in \mathbb Z^+$ to ensembles with $\beta = 2(r+1)$ \cite{Fo09}.
\end{rem}


\section{Joint Density of the Singular Values of Orthogonal Ensembles}\label{section:jointdensity}

In this section we assume that $w_1$ is an \emph{even} weight function supported on the interval $(-\omega,\omega)$.
By symmetry, we can establish the joint density of the singular values
by restricting ourselves to the cone of increasingly ordered singular
values
\begin{equation}\label{eqn:cone}
0 \leq \sigma_1 \leq \dotsb \leq \sigma_n,
\end{equation}
this way parametrizing  $|\OE_n(w_1)|$.
To simplify notation and to avoid case distinctions between odd and
even order $n$ in later parts of the paper, we
introduce two further sets of coordinates for this cone.
Writing, as detailed in \eqref{eqn:dim},
$n = 2m + \mu$ and $\hat{m}=m+\mu$ with $\mu=0,1$,
the coordinates
\begin{subequations}\label{eqn:coord_xy}
\begin{equation}\label{eqn:coord_xya}
x_j = \sigma_{2j-1}\quad (j=1,\ldots,\hat{m}), \qquad y_j = \sigma_{2j} \quad (j=1,\ldots,m)
\end{equation}
satisfy the interlacing property
\begin{equation}\label{eqn:interlacing_xy}
0 \leq x_1 \leq y_1 \leq x_2 \leq y_2 \leq \dotsb \leq x_{\hat{m}} \leq y_{\hat{m}}\leq \omega,
\end{equation}
\end{subequations}
with formally adding, if $\mu=1$, the value $y_{m+1}=\omega$. With
$x^\downarrow$ and $y^\downarrow$ denoting the $x$ and $y$ vectors
with their components taken in the reverse order, so
$x^\downarrow=(x_{\hat{m}},x_{\hat{m}-1},\dotsc,x_1)$ and
$y^\downarrow=(y_m,y_{m-1},\dotsc,y_1)$, we define, depending on the parity
of $n$, the coordinates
\begin{subequations}\label{eqn:coord_st}
\begin{equation} \label{eqn:coord_st_a}
(t,s)=(y^\downarrow,x^\downarrow) \quad (\mu=0),\qquad (t,s)=(x^\downarrow,y^\downarrow) \quad (\mu=1),
\end{equation}
satisfying the interlacing property
\begin{equation}\label{eqn:interlacing}
\omega\geq t_1 \geq s_1 \geq t_2 \geq s_2 \geq \cdots \geq t_{\hat m} \geq s_{\hat m} \geq 0,
\end{equation}
\end{subequations}
again formally adding the value $s_{m+1}=0$ if $\mu=1$.   
Since the
mapping from $\sigma=(\sigma_1,\ldots,\sigma_n)$ to either the pair of
coordinates $(x,y)$ or $(t,s)$ is orthogonal, transforming the density
between the three sets of coordinates is simply done by inserting new
variable names for old ones. Note that the $s$ variables parametrize the even-location decimated ensemble $\even |\OE_n(w_1)|$
while the $t$-variables do the same for $\odd|\OE_n(w_1)|$. We call them the even and odd singular values.

By the evenness of $w_1$ the joint probability density
of the singular values is, supported on \eqref{eqn:cone}, 
\[
q(\sigma_1,\ldots,\sigma_n) = n! \sum_{\epsilon \in \{\pm 1\}^n}
p(\epsilon_1 \sigma_1,\ldots,\epsilon_n \sigma_n) = c_{n,1} n! \cdot
\prod_{k=1}^n w(\sigma_k) \cdot D(\sigma_1,\ldots,\sigma_n)
\]
with 
\[
 D(\sigma_1,\ldots,\sigma_n) = \sum_{\epsilon \in \{\pm 1\}^n}
 |\Delta(\epsilon_1 \sigma_1,\ldots, \epsilon_n \sigma_n)|.
\]
Writing $D(x;y) = D(\sigma_1,\ldots,\sigma_n)$ in terms of $(x,y)$-coordinates, Bornemann and La Croix \cite[Eq.~(11)]{BLC} proved
in two different ways the algebraic fact
\[
D(x;y) = 2^n \cdot \Delta(x_1^2,\ldots,x_{\hat{m}}^2)\cdot y_1\cdots y_m \Delta(y_1^2,\ldots,y_,^2).
\]
Hence, we immediately get the following theorem: 
\begin{thm}\label{thm:JointDensity} Let $w_1$ be an even weight on $(-\omega,\omega)$. Then the joint probability density of $|\OE_n(w_1)|$,
supported on the cone~\eqref{eqn:interlacing_xy}, is given by
\begin{equation}\label{eqn:singularDensity}
q(x;y) = c_n \cdot \left( \prod_{k=1}^{\hat{m}} w_1(x_k) \cdot
\Delta(x_1^2,\ldots,x_{\hat{m}}^2) \right)\cdot \left( \prod_{k=1}^m
y_k w_1(y_k) \cdot \Delta(y_1^2,\ldots,y_{m}^2)\right)
\end{equation}
with $c_{n}= c_{n,1} n! 2^n$.
\end{thm}

\begin{rem} Because of the interlacing in \eqref{eqn:interlacing_xy}, this factorization does {\em not} reveal an independence
between the $x$ and $y$ variables.
\end{rem}


\section{Admissible Symmetric Weights}\label{section:admissible}

\begin{table}[tbp]
\caption{admissible symmetric weights $w_1(x)$}
{\small\begin{center}
\begin{tabular}{ccccccccc}
case & parameter & order & $w_1(x)$ & $\omega$ & $\alpha_k$ & $\beta_k$ & $\phi(x)$ & $\theta$\\*[1mm]\hline
Gauss$\phantom{\Bigg|}$ & --- & $\kappa<\infty$ & $e^{-x^2/2}$& $\infty$ & $1$ & $k-1$ & $1$ & $\sqrt{\displaystyle\frac{\pi}{2}}$\\
Jacobi$\phantom{\Bigg|}$& $a >-1$& $\kappa<\infty$ & $(1-x^2)^a$ & $1$ & $\dfrac{1}{2a+1+k}$ & $\dfrac{k-1}{2a+1+k}$ & $1-x^2$ & $\displaystyle\frac{\sqrt{\pi}\,\Gamma(a+1)}{2\Gamma(a+\frac{3}{2})}$\\
Cauchy$\phantom{\Bigg|}$& $a>-\dfrac12$ & $\kappa<2a$ & $(1+x^2)^{-a-1}$ & $\infty$ & $\dfrac{1}{2a+1-k}$ & $\dfrac{k-1}{2a+1-k}$ & $1+x^2$& $\displaystyle\frac{\sqrt{\pi}\,\Gamma(a+\frac12)}{2\Gamma(a+1)}$\\*[1mm]\hline
\end{tabular}
\end{center}}
\label{tab:admissible}
\end{table}

We call a smooth integrable weight $w_1: (-\omega,\omega) \to (0,\infty)$ {\em admissible} of order $\kappa$ and
mass
\begin{equation}\label{eq:theta}
2\theta = \int_{-\omega}^\omega w_1(\xi)\,d\xi,
\end{equation}
if it
satisfies the following properties
\begin{itemize}
\item[(a)] $w_1$ is even;
\item[(b)] $w_1$ is normalized: $w_1(0)=1$;
\item[(c)] $w_1$ satisfies a three-term recurrence of antiderivatives of the form
\begin{equation}\label{eqn:recursion}
\int^x \xi^{k} w_1(\xi) \,d\xi = -\alpha_{k}  x^{k-1}  \phi(x) w_1(x)  + \beta_k \int^x \xi^{k-2}   w_1(\xi) \,d\xi \qquad (k=1,2,\ldots,\kappa),
\end{equation}
with a smooth function $\phi:(-\omega,\omega)\to (0,\infty)$ and constants $\alpha_k, \beta_k$ such that $\beta_1=0$;
\item[(d)] $w_1$ vanishes at the boundary:
\[
\lim_{x\to\omega} x^k w_1(x) = \lim_{x\to\omega} x^k \phi(x) w_1(x) = 0 \qquad (k=0,1,\ldots,\kappa).
\]
\end{itemize}
Table~\ref{tab:admissible} lists three
cases of such admissible weights; by Theorem~\ref{thm:classification} below, these are {\em all} possible cases.

By defining
$\alpha_0=1$, $\beta_0=0$ and
\[
\psi(x) = - \frac{1}{\phi(x) w_1(x)}\int_0^x w_1(\xi)\,d\xi\qquad (-\omega < x < \omega),
\]
the recurrence (c) extends to the case $k=0$ if we replace $x^{-1}$ by $\psi(x)$. 
By introducing the vectors
\[
\pi_\nu^n(x) = 
\begin{pmatrix}
x^{\nu} \\
x^{\nu +2}\\
\vdots\\
x^{\nu +2n-2}
\end{pmatrix}
  \in \R^n \qquad (\nu = -1,0,1),
\]
with the understanding that, instead of $x^{-1}$, the first entry of $\pi_{-1}^n(x)$ is in fact $\psi(x)$, we
can write the thus extended recurrence in 
compact matrix-vector form
\begin{subequations}\label{eqn:compactrecurrence}
\begin{align}
\int^x w_1(\xi) \pi_\nu^n(\xi) \,d\xi &=  L_{n,\nu}\cdot \tilde{w}_1(x) \pi_{\nu-1}^n (x) \qquad (\nu=0,1,\;\; 2n+\nu \leq \kappa+2),\\
 \tilde{w}_1(x) &= \phi(x)\,w_1(x),
\end{align}
\end{subequations}
with a constant \emph{lower triangular} matrix $L_{n,\nu} \in {\R^{n\times n}}$ having the numbers $-\alpha_\nu,-\alpha_{\nu+2},\ldots,-\alpha_{\nu+2n-2}$ along its main diagonal.
In particular, there holds
\begin{equation}\label{eqn:Ldet}
\det L_{n,\nu} = (-1)^n  A_{n,\nu},\qquad A_{n,\nu} = \prod_{k=0}^{n-1} \alpha_{2k+\nu}.
\end{equation}
Since within the range of $k$ restricted by the order $\kappa$ the constants $\alpha_k$ and
$\beta_k$ given in Table~\ref{tab:admissible} are strictly positive (with the exception of $\beta_1=0$), we have $A_{n,\nu}>0$. We call $\tilde{w}_1=\phi w_1$ the {\em companion weight} of $w_1(x)$ and observe that
\begin{equation}\label{eqn:thetalim}
\lim_{s\to\omega} \tilde{w}_1(s)\psi(s) = -\theta.
\end{equation}

In analogy to the results recalled in Section~\ref{subsect:Hermitian}, we have the following \emph{uniqueness} result.

\begin{thm}\label{thm:classification}
Up to a rescaling of $x$, all possible admissible weights $w_1(x)$ are listed in Table~\ref{tab:admissible}. Actually,
properties (b)--(d) of an admissible weight are sufficient for the conclusion to hold, that is, those properties already imply the evenness assumption
(a).
\end{thm}

\begin{proof} Let $w_1(x)$ be an admissible weight. Differentiating \eqref{eqn:recursion} yields 
\begin{equation}\label{eqn:ode1}
(x^2 - \beta_k + (k-1)\alpha_k \phi) w_1 = -\alpha_k x (\phi w_1)'\qquad (k=1,2,\ldots,\kappa).
\end{equation}
Inserting $x=0$ gives
\[
\beta_k = (k-1) \alpha_k \phi(0).
\]
Therefore, if $\alpha_k=0$ for some positive integer $k$, we would get also that $\beta_k=0$ and, hence, that $\int^x \xi^k w_1(\xi)\,d\xi = 0$
in contradiction to $w_1$ being positive. We conclude that
\[
\alpha_k \neq 0 \qquad (k=1,2,\ldots,\kappa).
\]
Inserting $k=1$ into \eqref{eqn:ode1} gives the differential equation
\begin{equation}\label{eqn:ode2}
(\phi w_1)' = -\frac{1}{\alpha_1} x w_1,\qquad w_1(0)=1.
\end{equation}
Inserting this expression for $(\phi w_1)'$ into \eqref{eqn:ode1} and rearranging, we get
\[
\frac{x^2-\beta_k}{\alpha_k} + (k-1)\phi = \frac{x^2}{\alpha_1}\qquad (k=1,2,\ldots,\kappa).
\]
Solving for $\phi$ gives
\[
\phi(x) = \frac{1}{k-1}\frac{\alpha_k-\alpha_1}{\alpha_k\alpha_1} x^2 + \frac{1}{k-1}\frac{\beta_k}{\alpha_k}
= \frac{1}{k-1}\frac{\alpha_k-\alpha_1}{\alpha_k\alpha_1} x^2 + \phi(0).
\]
Since $\phi(x)$ is assumed to be \emph{independent} of $k$, we get
that
\begin{equation}\label{eqn:phi}
\phi(x) = \phi(0) + \tau x^2,\qquad 
\tau = \frac{\alpha_2-\alpha_1}{\alpha_2\alpha_1} = \frac{1}{k-1}\frac{\alpha_k-\alpha_1}{\alpha_k\alpha_1}  
\quad (k=2,3,\ldots,\kappa),
\end{equation}
which can be solved for $\alpha_k$:
\[
\alpha_k = \frac{\alpha_1\alpha_2}{\alpha_2 + (k-1)(\alpha_1-\alpha_2)}.
\]
Now, we distinguish four cases depending on whether $\phi(0)$ and $\tau$ are zero or not.

\paragraph{Case 1} $\phi(0)=0$ and $\tau=0$, that is, $\phi\equiv0$. By \eqref{eqn:ode2} $x w_1 \equiv 0$, which contradicts the positivity of $w_1$.

\paragraph{Case 2} $\phi(0)=0$ and $\tau\neq 0$. By absorbing a rescaling of the
$\alpha_k$ into $\phi$ we can arrange for $\phi(x) = \pm x^2$.
Now, solving the differential equation \eqref{eqn:ode2} for $w_1$ yields
\[
w_1(x) = c x^{-2\mp\frac{1}{\alpha_1}}
\]
with some constant $c$. For $w_1(0)=1$ to make sense, we would need the exponent to vanish, implying that already $w_1 \equiv 1$.
But such a weight would not satisfy $w_1(x)\to 0$ as $x\to\omega$. 

\paragraph{Case 3} $\phi(0)\neq 0$ and $\tau=0$. By rescaling $x$ we can arrange for
$\alpha_1=\pm 1$ and $\phi\equiv 1$. Now, solving the initial value problem \eqref{eqn:ode2} for $w_1$ yields
\[
w_1(x) = e^{\mp x^2/2}.
\]
From $w_1(x)\to 0$ as $x\to\omega$ we get $\alpha_1=1$ and $\omega=\infty$. This yields
 the Gauss case of Table~\ref{tab:admissible}.

\paragraph{Case 4} $\phi(0)\neq 0$ and $\alpha_1\neq \alpha_2$. By rescaling $x$ and absorbing a rescaling of the $\alpha_k$ into $\phi$ we can arrange for $\phi(x) = 1 \pm x^2$.
 Now, solving the initial value problem \eqref{eqn:ode2} for $w_1$ yields
\[
w_1(x) = (1\pm x^2)^{-1\mp \frac{1}{2\alpha_1}}
\]
In the case $\phi(x)=1-x^2$ we set $a=-1+\frac{1}{2\alpha_1}$ and get, assuring integrability,
\[
w_1(x) = (1-x^2)^{a},\quad a>-1,\qquad \omega=1,\qquad\alpha_1 = \frac{1}{2a+2},
\]
which yields the symmetric Jacobi case of Table~\ref{tab:admissible}.
In the case $\phi(x) = 1+x^2$ we set $a=\frac{1}{2\alpha_1}$ and get, once more assuring integrability,
\[
w_1(x) = (1+x^2)^{-a-1},\qquad a>-\frac{1}{2},\qquad \omega=\infty,\qquad \alpha_1 = \frac{1}{2a},
\]
which finally yields the Cauchy case of Table~\ref{tab:admissible} (the only case where there is a restriction of
the maximum order $\kappa$ that has to be checked).
\end{proof}

\begin{rem} 
If $\phi(0)\neq 0$, Eqs.~\eqref{eqn:ode2} and \eqref{eqn:phi} imply that the logarithmic derivative of $\tilde w_1 = \phi w_1$,
namely 
\[
\frac{\tilde w_1'}{\tilde w_1} = - \frac{x}{\alpha_1\phi(x)},
\]
takes the form of a ratio of a linear and a quadratic polynomial.
Hence,  we immediately see that $\tilde w_1$ must be a \emph{classical} weight. Because of the common denominator~$\phi$
in the logarithmic derivatives, the same conclusion holds for the weights $w_1$ and $w_2 = w_1 \tilde w_1$. 
Therefore, we could have finished the proof by checking properties (b)--(d) for each entry of a list of all classical weights. 
\end{rem}

\section{Integrating Out the Odd and Even Singular Values}\label{section:evensingularvalues}

\subsection{Integrating out the odd singular values}
We now prove Theorem~\ref{thm:main1}.
To begin with, we transform the joint
density (\ref{eqn:singularDensity})  to $(s,t)$ coordinates, that is,
 \[
 q(s;t) = c_n \cdot 
 g_{\mu}(s_1,\ldots,s_m) \cdot g_{1-\mu}(t_1,\ldots,t_{\hat{m}})
\]
with functions
\begin{equation}\label{eqn:gmu}
  g_{\nu}(z_1,\ldots,z_m) = \prod_{k=1}^m z_k^\nu w_1(z_k) \cdot \Delta(z_m^2,\ldots,z_1^2).
\end{equation}
Likewise, we write $\tilde{g}_\nu$ for the same form of expression using the companion weight $\tilde w_1$ instead of $w_1$.

Now, Corollary~\ref{cor:2} below shows that integrating out the odd singular
values $t$ subject to the interlacing~\eqref{eqn:interlacing} gives
the following marginal density of the even singular values:
\begin{multline}\label{eqn:even_density}
q_{\text{even}}(s_1,\ldots,s_m) = 
c_n \theta^\mu A_{\hat{m},1-\mu} \cdot g_\mu(s_1,\ldots,s_m)  \tilde{g}_\mu(s_1,\ldots,s_m)\\*[2mm]
= c_n \theta^\mu A_{\hat{m},1-\mu}  \cdot\prod_{k=1}^m s_k^{2\mu} w_2(s_k) \cdot \Delta(s_m^2,\ldots,s_1^2)^2\qquad (\omega\geq s_1 \geq \cdots \geq s_m \geq 0)
\end{multline}
defining the associated weight function (cf. \cite[Remark on p.~186]{FR})
\begin{equation}\label{eqn:w2}
w_2(s) = w_1(s) \hat{w_1}(s) = \phi(s) w_1(s)^2.
\end{equation}
Since the last expression in \eqref{eqn:even_density} is easily identified as the joint density of $\chUE_m(x^{2\mu}w_2)$, see Eq.~\eqref{eqn:chiraldensity}, we have finally proved Theorem~\ref{thm:main1}.

\begin{rem} As a side product, the representation \eqref{eqn:even_density} shows that the normalization constant $c_{m,\mu}^{\text{ch}}$ of 
the joint density of $\OE(x^{2\mu}w_1)$, if extended by symmetry to be supported on $(0,\infty)^m$, is given by
\[
c_{m,\mu}^{\text{ch}} = c_{n,1}A_{\hat{m},1-\mu}\theta^\mu\frac{2^n n! }{m!}.
\]
\end{rem}

The integration is based on the following lemma and its first Corollary~\ref{cor:2}.

\begin{lemma}\label{lem:1} Let $\tilde{w}_1$ be the companion of the admissible weight $w_1$.
Then, there holds 
\begin{multline*}
\int_{x_1}^{x_2}d \xi_1 \cdots \int_{x_n}^{x_{n+1}}d \xi_n \;\det\left(w_1(\xi_1)\pi_\nu ^n(\xi_1) \; \cdots \; w_1(\xi_n)\pi_\nu^n(\xi_n)\right)\\*[2mm]
 = A_{n,\nu}
\det
\begin{pmatrix}
\tilde{w}_1(x_1)\pi_{\nu-1}^n(x_1) & \cdots & \tilde{w}_1(x_{n+1})\pi_{\nu-1}^n(x_{n+1}) \\*[1mm]
1 & \cdots & 1
\end{pmatrix}\qquad (\nu=0,1).
\end{multline*}
Here, all integration bounds are within $(0,\omega)$ and, in the case of a Cauchy weight,  $2n+\nu \leq \kappa+2$. 
\end{lemma}

\begin{proof} Simplifying the notation to $\pi_\nu(x) = \pi^n_\nu(x)$, we calculate by means of \eqref{eqn:compactrecurrence}
and \eqref{eqn:Ldet}
\begin{multline*}
\int_{x_1}^{x_2}d \xi_1 \cdots \int_{x_n}^{x_{n+1}}d \xi_n \;\det\left(w_1(\xi_1)\pi_\nu(\xi_1) \; \cdots \; w_1(\xi_n)\pi_\nu(\xi_n)\right)\\*[2mm]
 = \det\left(\int_{x_1}^{x_2} w_1(\xi_1)\pi_\nu(\xi_1) d\xi_1 \; \cdots \; \int_{x_n}^{x_{n+1}} w_1(\xi_n)\pi_\nu(\xi_n) d\xi_{n}\right)
 \end{multline*}
 \begin{multline*}
=\underbrace{ \det L_{n,\nu}}_{=(-1)^n A_{n,\nu}}  \cdot 
\det\left(\tilde{w}_1\pi_{\nu-1}\Big|_{x_1}^{x_2}\;\; \cdots \;\; \tilde{w}_1\pi_{\nu-1}\Big|_{x_n}^{x_{n+1}}\right)\\*[2mm]
= A_{n,\nu}\det
\begin{pmatrix}
\tilde{w}_1(x_1)\pi_{\nu-1}(x_1) & \tilde{w}_1\pi_{\nu-1}\Big|_{x_1}^{x_2} & \cdots & \tilde{w}_1 \pi_{\nu-1}\Big|_{x_n}^{x_{n+1}} \\*[3mm]
1 & 0 & \cdots & 0
\end{pmatrix}\\*[2mm]
=  A_{n,\nu} \det
\begin{pmatrix}
\tilde{w}_1(x_1)\pi_{\nu-1}(x_1) & \tilde{w}_1(x_2)\pi_{\nu-1}(x_2) &\cdots & \tilde{w}_1(x_{n+1})\pi_{\nu-1}(x_{n+1}) \\*[2mm]
1 & 1 & \cdots & 1
\end{pmatrix}.
\end{multline*}
In the last step we added the first column to the second, then
 the second to the third, etc.
\end{proof}

\begin{cor}\label{cor:2} Let $g_{\nu}$, $\tilde{g}_\nu$ be as in \eqref{eqn:gmu} and put $s_{\hat m}=0$ if $\mu=1$.
Then, there holds 
\begin{equation}\label{DAV}
\int_{s_1}^{\omega} dt_1 \int_{s_2}^{s_1} dt_2\cdots \int_{s_{\hat m}}^{s_{\hat{m}-1}} dt_{\hat{m}}\, g_{1-\mu}(t_1,\ldots,t_{\hat m}) =\theta^{\mu} A_{\hat{m},1-\mu}\cdot \tilde{g}_{\mu}(s_1,\ldots,s_m)\qquad (\mu=0,1).
\end{equation}
Here, all integration bounds are within $(0,\omega)$ and, in the case of a Cauchy weight,  $n=2m+\mu \leq \kappa+2$. 
\end{cor}

\begin{proof} Using the notation of Lemma~\ref{lem:1}, we first observe that 
\begin{equation}\label{eqn:gmu_as_det}
g_{\mu}(z_1,\ldots,z_m) = \det\left(w_1(z_m) \pi^m_\mu(z_m)\;\cdots\; w_1(z_1)\pi^m_\mu(z_1)\right)
\end{equation}
and the same for $\tilde{g}_\mu$ with weight $\tilde{w}_1$.
Now, Lemma~\ref{lem:1} yields, first using $\tilde{w}_1(s) \pi_0^m(s)\to 0$ as $s\to\omega$, that for $\mu=0$
\begin{multline*}
\int_{s_1}^{\omega} dt_1 \int_{s_2}^{s_1} dt_2\cdots \int_{s_m}^{s_{m-1}} dt_m\, \,\det\left(w_1(t_m)\pi_1^m(t_m)\;\cdots\; w_1(t_1) \pi_1^m(t_1)\right)\\*[2mm]
 = 
A_{m,1}\det
\begin{pmatrix}
\tilde{w}_1(s_m)\pi_{0}^m(s_m) & \cdots & \tilde{w}_1(s_1)\pi_{0}^m(s_{1}) & 0 \\*[1mm]
1 &  \cdots & 1 & 1
\end{pmatrix}\\*[2mm]
 = 
A_{m,1}\det\left(\tilde{w}_1(s_m)\pi_{0}^m(s_m)\;\cdots\;\tilde{w}_1(s_1)\pi_{0}^m(s_{1})\right)
\end{multline*}
and then, using $\pi_{-1}^{m+1}(0)=0$ and $\tilde{w}_1(s)\psi(s) \to -\theta$ as $s\to\omega$, that for $\mu=1$
\begin{multline*}
\int_{s_1}^{\omega} dt_1 \int_{s_2}^{s_1} dt_2\cdots \int_{0}^{s_m} dt_{m+1} \, \det\left(w_1(t_{m+1})\pi_0^{m+1}(t_{m+1})\;\cdots\; w_1(t_1)\pi_0^{m+1}(t_1)\right) \\*[2mm]
 = A_{m+1,0}\cdot
\det
\begin{pmatrix}
0 & \tilde{w}_1(s_m)\pi_{-1}^{m+1}(s_m) & \cdots & \tilde{w}_1(s_1)\pi_{-1}^{m+1}(s_{1}) & \lim_{s\to\omega} \tilde{w}_1(s)\pi_{-1}^{m+1}(s)\\*[1mm]
1 & 1 & \cdots & 1 & 1
\end{pmatrix}\\*[2mm]
=
(-1)^m A_{m+1,0} \cdot\det
\begin{pmatrix}
 \tilde{w}_1(s_m)\pi_{-1}^{m+1}(s_m) & \cdots & \tilde{w}_1(s_1)\pi_{-1}^{m+1}(s_1) & \lim_{s\to\infty}\tilde{w}_1(s) \pi_{-1}^{m+1}(s)
\end{pmatrix}\\*[2mm]
=
(-1)^m A_{m+1,0}\cdot\det
\begin{pmatrix}
\tilde{w}_1(s_m) \psi(s_m) & \cdots & \tilde{w}_1(s_1)\psi(s_1) & -\theta \\*[1mm]
\tilde{w}_1(s_m) \pi_{1}^m(s_m) & \cdots & \tilde{w}_1(s_1) \pi_{1}^m(s_1) & 0
\end{pmatrix}\\*[2mm]
 = \theta A_{m+1,0}\cdot
 \det\left(\tilde{w}_1(s_m) \pi_{1}^m(s_m)\;\cdots\; \tilde{w}_1(s_1) \pi_{1}^m(s_1)\right),
\end{multline*}
which finishes the proof.
\end{proof}

\begin{rem}
In the Jacobi case the multidimensional integral (\ref{DAV}) can be recognized as a variant of the Dixon--Anderson
integral \cite{Di05,An91}, well known in the theory of the Selberg integral, and also in the theory
of $\beta$-ensembles in random matrix theory \cite[Section 4.2]{Fo10}. Specifically, in the
statement of the Dixon--Anderson integral given in \cite[Eq.~(4.15)]{Fo10}, cf. \cite[Eq.~(6)]{Di05}, that is, 
\[
\int_{x_1}^{x_0}d\xi_1 \cdots \int_{x_{\hat m}}^{x_{\hat{m}-1}}d\xi_{\hat{m}} \, \Delta(\xi_{\hat{m}},\ldots,\xi_1) 
\prod_{j=1}^{\hat{m}}\prod_{k=0}^{\hat{m}} |\xi_j - x_k|^{a_k-1} = \frac{\prod_{i=0}^{\hat{m}}\Gamma(a_i)}{\Gamma(\sum_{i=0}^{\hat{m}} a_i)} \prod_{0\leq j < k \leq \hat{m}} (x_j-x_k)^{a_j+a_k-1},
\]
valid for $x_0 > x_1 > \cdots > x_{\hat{m}}$ and $a_j > 0$ ($j=0,\ldots,\hat{m}$),
we can reclaim the Jacobi case of (\ref{DAV}) by the following substitutions of variables and choices of parameters:
\[
x_0 = 1,\quad x_j = s_j^2, \quad \xi_j = t_j^2 \quad (j=1,\ldots,\hat{m}); \qquad a_0 = a+1, \quad a_j = 1 \quad (j=1,\ldots,m),
\]
and $a_{m+1} = 1/2$, $s_{m+1}=0$ if $\mu=1$.
\end{rem}

\subsection{Integrating out the even singular values} The following second corollary of Lemma~\ref{lem:1} will allow us to integrate out the \emph{even} singular values from the density $q(s;t)$.

\begin{cor}\label{cor:even_out} Let $g_\mu$, $\tilde g_\mu$ be as in \eqref{eqn:gmu} and put $t_{m+1}=0$ if $\mu=0$.
Then, there holds
\[
\int_{t_2}^{t_1} ds_1 \cdots \int_{t_{m+1}}^{t_m} ds_{m}\, g_{\mu}(s_1,\ldots,s_m)
= A_{m,\mu}\det
\begin{pmatrix}
\tilde w_1(t_{\hat m}) \pi_{1-\mu}^{\hat{m}-1}(t_{\hat{m}}) &\cdots & \tilde w_1(t_1)\pi_{1-\mu}^{\hat{m}-1}(t_1) \\*[2mm]
\theta_{1-\mu}(t_{\hat{m}})  & \cdots & \theta_{1-\mu}(t_1)
\end{pmatrix}
\]
for $\mu=0,1$ with $\theta_0(x)=1$ and $$\theta_1(x) =  \displaystyle\int_0^{x} w_1(\xi)\,d\xi.$$
Here, all integration bounds are within $(0,\omega)$ and, in the case of a Cauchy weight,  $n=2m+\mu \leq \kappa+2$.
\end{cor}

\begin{proof} Using (\ref{eqn:gmu_as_det}) and Lemma~\ref{lem:1} we obtain
\begin{multline*}
\int_{t_2}^{t_1} ds_1 \cdots \int_{t_{m+1}}^{t_m} ds_{m}\, g_{\mu}(s_1,\ldots,s_m)\\*[2mm] =
\int_{t_2}^{t_1} ds_1 \cdots \int_{t_{m+1}}^{t_m} ds_{m}\,
\det\left(w_1(s_m) \pi^m_\mu(s_m)\;\cdots\; w_1(s_1)\pi^m_\mu(s_1)\right)\\*[2mm]
= A_{m,\mu}\det
\begin{pmatrix}
\tilde w_1(t_{m+1}) \pi_{\mu-1}^m(t_{m+1}) & \cdots & \tilde w_1(t_1)\pi_{\mu-1}^m(t_1) \\*[1mm]
1 & \cdots & 1
\end{pmatrix},
\end{multline*}
which is already the assertion for $\mu=1$. For $\mu=0$, the assertion follows from further calculating
\begin{multline*}
\det
\begin{pmatrix}
\tilde w_1(t_{m+1}) \pi_{-1}^m(t_{m+1}) & \tilde{w}_1(t_m) \pi_{-1}^m(t_{m}) & \cdots & \tilde w_1(t_1)\pi_{-1}^m(t_1) \\*[1mm]
1 & 1 & \cdots & 1
\end{pmatrix}\\*[2mm]
=
\det
\begin{pmatrix}
0 & \tilde w_1(t_m) \pi_{-1}^m(t_{m}) & \cdots & \tilde w_1(t_1)\pi_{-1}^m(t_1) \\*[1mm]
1 & 1 & \cdots & 1
\end{pmatrix}
\end{multline*}
\begin{multline*}
= (-1)^m \det\left(\tilde w_1(t_m) \pi^m_{-1}(t_m)\;\cdots\; \tilde{w}_1(t_1)\pi_{-1}^m(t_1)\right)\\*[2mm]
 =
(-1)^m\det
\begin{pmatrix}
\tilde{w}_1(t_m) \psi(t_m) & \cdots & \tilde w_1(t_1)\psi(t_1)  \\*[1mm]
\tilde{w}_1(t_m) \pi_{1}^{m-1}(t_m) & \cdots & \tilde{w}_1(t_1)\pi_{1}^{m-1}(t_1)
\end{pmatrix}\\*[2mm]
= \det
\begin{pmatrix}
\tilde{w}_1(t_m) \pi_{1}^{m-1}(t_m) & \cdots & \tilde{w}_1(t_1)\pi_{1}^{m-1}(t_1)\\*[1mm]
-\tilde{w}_1(t_m) \psi(t_m) & \cdots & -\tilde w_1(t_1)\psi(t_1)  \\*[1mm]
\end{pmatrix}
\end{multline*}
which finishes the proof.
\end{proof}
Now, by means of this corollary, the marginal density of the odd singular values is given as  
\begin{multline*}
q_{\text{odd}}(t_1,\ldots,t_{\hat{m}}) =
c_n A_{m,\mu} \cdot g_{1-\mu}(t_{\hat{m}},\ldots,t_1) \cdot
\det
\begin{pmatrix}
\tilde w_1(t_{\hat m}) \pi_{1-\mu}^{\hat{m}-1}(t_{\hat{m}}) &\cdots & \tilde w_1(t_{\hat m})\pi_{1-\mu}^{\hat{m}-1}(t_1) \\*[2mm]
\theta_{1-\mu}(t_{\hat{m}})  & \cdots & \theta_{1-\mu}(t_1)
\end{pmatrix}\\*[2mm]
=
c_n A_{m,\mu} \cdot 
\det
\begin{pmatrix}
\tilde w_1(t_{\hat m}) \pi_{1-\mu}^{\hat{m}-1}(t_{\hat{m}}) &\cdots & \tilde w_1(t_{\hat m})\pi_{1-\mu}^{\hat{m}-1}(t_1) \\*[2mm]
\gamma_{1-\mu}(t_{\hat{m}})  & \cdots & \gamma_{1-\mu}(t_1)
\end{pmatrix}\\*[2mm]
\cdot
\det
\begin{pmatrix}
\tilde w_1(t_{\hat m}) \pi_{1-\mu}^{\hat{m}-1}(t_{\hat{m}}) &\cdots & \tilde w_1(t_{\hat m})\pi_{1-\mu}^{\hat{m}-1}(t_1) \\*[2mm]
\theta_{1-\mu}(t_{\hat{m}})  & \cdots & \theta_{1-\mu}(t_1)
\end{pmatrix}\qquad (\omega\geq t_1 \geq \cdots \geq t_{\hat{m}} \geq 0)
\end{multline*}
with
\[
\gamma_\mu(x) = \tilde w_1(x) x^{\mu+2\hat{m}-2},\qquad
\theta_\mu(x) = 
\begin{cases}
1 &  \quad\text{if $\mu=0$,} \\*[2mm]
 \displaystyle\int_0^x w_1(\xi)\,d \xi & \quad\text{if $\mu=1$.}
\end{cases}
\]
Note that the two determinantal factors differ just in their last rows. It is this difference that prevents the expression from
becoming a perfect square, which is in marked contrast with the marginal density \eqref{eqn:even_density} of the even singular values. 

\begin{table}[tbp]
\caption{companion weights $\tilde w_1(x)$ and integrals $\theta_1(x)$}
{\small\begin{center}
\begin{tabular}{lccc}
type & $w_1(x)$ & $\tilde w_1(x)$ & $\theta_1(x)$\\*[1mm]\hline
Gauss$\phantom{\Bigg|}$  & $e^{-x^2/2}$& $e^{-x^2/2}$ & $\sqrt{\dfrac{\pi}{2}}\erf\left(\dfrac{x}{\sqrt{2}}\right)$\\
Jacobi$\phantom{\Bigg|}$ & $(1-x^2)^a$ & $(1-x^2)^{a+1}$ & $x \cdot{}_2 F_1\!\left(\dfrac12,-a;\dfrac32;x^2\right)$\\
Cauchy$\phantom{\Bigg|}$ & $(1+x^2)^{-a-1}$ & $(1+x^2)^{-a}$ & $x \cdot{}_2 F_1\!\left(\dfrac12,a+1;\dfrac32;-x^2\right)$\\*[1mm]\hline
\end{tabular}
\end{center}}
\label{tab:theta_function}
\end{table}%


\section{Gap Probabilities}\label{section:gapprobabilities}

\subsection{A corollary of Theorem~\ref{thm:main1}}
Theorem~\ref{thm:main1} has an interesting implication in terms of gap probabilities, a notion that we recalled in Section~\ref{section:revision}. Specifically, we get the following result:
\begin{thm}\label{thm:gap} The gap probabilities of the Gauss, symmetric Jacobi or Cauchy case of Table~\ref{tab:admissible1} of
order $n=2m+\mu$ satisfy
\[
E_{n,1}(2k+\mu-1;(-s,s);w_1) + E_{n,1}(2k+\mu;(-s,s);w_1) = E_{m,2}(k;(0,s^2);x^{\mu-\frac12}w_2(x^{1/2})\chi_{x>0}).
\]
\end{thm}
\begin{proof}
The change of variables $x_k \mapsto \tilde x_k = \sqrt{x_k}$, applied to the
joint density $p_\text{ch}$ of the chiral ensemble $\chUE(x^{2\mu}w_2(x))$ yields
\[
p_\text{ch}(x_1,\ldots,x_m) \,dx_1\cdots dx_m = p_{m,2}(\tilde x_1,\ldots,\tilde x_m) \,d\tilde x_1 \cdots d\tilde x_m
\]
where $p_{m,2}$ is the density of $\UE(x^{\mu-\frac12} w_2(x^{1/2}) \chi_{x>0})$. Hence, lifted to gap probabilities, we obtain
\begin{equation}\label{eqn:chUE}
E_{m,2}(k;(0,s);\chUE(x^{2\mu }w_2)) = E_{m,2}(k;(0,s^2);x^{\mu-\frac12}w_2(x^{1/2})\chi_{x>0})\qquad (\mu=0,1).
\end{equation}
By looking at pairs of consecutive values it is easy to see that the event that exactly $k$ values of the decimated ensemble $\even|\OE_{n}(w_1)|$, $n=2m+\mu$, are contained in $(0,s)$ is given by the union of the events that
exactly $2k+\mu-1$ or that exactly $2k+\mu$ values of $|\OE_{n}(w_1)|$ are in that interval. Since these two events are {\em mutually exclusive}
and since the singular values of $\OE_n$ contained in $(0,s)$ correspond to the eigenvalues
in $(-s,s)$, we thus get from \eqref{eqn:even_aGUE}
and (\ref{eqn:chUE}) proof of 
\begin{multline}\label{eqn:gap}
E_{n,1}(2k+\mu-1;(-s,s);w_1) + E_{n,1}(2k+\mu;(-s,s);w_1)  = E_{m,2}(k;(0,s);\chUE(x^{2\mu }w_2))  \\*[2mm]
= E_{m,2}(k;(0,s^2);x^{\mu-\frac12}w_2(x^{1/2})\chi_{x>0}),
\end{multline} 
which finishes the proof.
\end{proof}

For even order ($\mu=0$), a first proof of this theorem was given by Forrester \cite[Eq.~(1.14)]{EvenSymm}
using generating functions, Pfaffians and Fredholm determinants. For Gaussian weights, Bornemann and La Croix \cite[Eq.~(40)]{BLC}
recently settled the odd order case by using the more elementary techniques similar to this paper.

\subsection{Alternative derivation of Theorem \ref{thm:gap}}

A natural question is to enquire if the proof of Theorem \ref{thm:gap} for $\mu=0$ given in \cite{EvenSymm} can be extended to the case $\mu = 1$.
Here we will show that the answer is yes, although as is usual for methods based on Pfaffians in the study of random matrix ensembles with $\beta =1$
(see e.g.~\cite[Section 6.3.3]{Fo10}),
the number of eigenvalues being odd adds to the complexity of the calculation.

The first step is to introduce the generating function of the gap probabilities $\{ E_{n,\beta}(k;J;w_\beta) \}$ according to
$$
E_{n,\beta}(J;\xi;w_\beta) = \sum_{k=0}^\infty (1 - \xi)^k E_{n,\beta}(k;J;w_\beta).
$$
The generating function can be expressed as the multidimensional integral (see e.g.~\cite[Prop.~8.1.2]{Fo10})
\begin{equation}\label{E}
E_{n,\beta}(J;\xi;w_\beta)  = 
\int_{-\omega}^\omega dx_1 \cdots  \int_{-\omega}^\omega dx_n \, \prod_{j=1}^n (1 - \xi \chi_{x_j \in J} ) \cdot p_\beta(x_1,\dots, x_n).
\end{equation}
In terms of generating functions, the assertion of Theorem  \ref{thm:gap}  in the case $\mu = 1$ is equivalent to
\begin{multline}\label{B1}
\bigg ( {1 \over (2k)!} {\partial^{2k} \over \partial \xi^{2k} }-
{1 \over (2k+1)!} {\partial^{2k+1} \over \partial \xi^{2k+1} } \bigg )
E_{2m+1,1}((-s,s);\xi; w_1(x)) \Big |_{\xi = 1} \\ =  {(-1)^k \over k!} {\partial^{k} \over \partial \xi^{k} }  E_{m,2}((0,s^2); \xi;x^{1/2} w_2(x^{1/2}) \chi_{x > 0})  \Big |_{\xi = 1}
\end{multline}
being valid for the weights in Table \ref{tab:admissible1}. It is this identity that we prove in the rest of the section. 

By making use of Pfaffians, (\ref{E}) for $\beta = 1$ and $w_1(x)$ even can be expressed as a determinant.

\begin{lemma}
Let $R_j(x)$ be a polynomial of degree $j$ for each $j=0,1,\dots$, and furthermore require that $R_j(x)$ be even (odd) for $j$ even (odd).
For $w_1(x)$ even we have
\begin{equation}\label{Pr}
E_{2m+1,1}((-s,s);\xi;w_1) \;\propto\; \det Y,
\end{equation}
where
\begin{equation}\label{Y}
Y = \Big ( [ a_{2j-1,2k} ]_{j=1,\dots,m+1 \atop k = 1,\dots,m} \: [b_{2j-1} ]_{j=1,\dots,m+1} \Big )
\end{equation}
with
\begin{align}\label{0.7}
a_{j,k} = & {1 \over 2} \int_{-\omega}^\omega dx \, w_1(x) (1 - \xi \chi_{x \in (-s,s)})
\int_{-\omega}^\omega dy \, w_1(y) (1 - \xi \chi_{y \in (-s,s)}) \, \nonumber \\*[2mm]
& 
\times R_{j-1}(x) {\rm sgn} \, (y-x)  R_{k-1}(y),  \nonumber \\*[3mm]
b_j  = & {1 \over 2}  \int_{-\omega}^\omega  w_1(x)(1 - \xi \chi_{x \in (-s,s)}) R_{j-1}(x) \, dx.
\end{align}
The proportionality in (\ref{Pr}) is such that the RHS is equal to unity when $\xi = 0$.
\end{lemma}

\noindent
\begin{proof} Let $h(x,y) = - h(y,x)$ and set
\begin{equation}\label{X}
X = \left ( \begin{array}{cc}  [h(x_j,x_k)]_{j,k=1,\dots,2m+1} &
[F(x_j)]_{j=1,\dots,2m+1}
\\*[2mm] - [F(x_k)]_{k=1,\dots,2m+1} & 0
\end{array}
\right ).
\end{equation}
It is well known (see e.g.~\cite[Eq.~(6.81)]{Fo10}) that with $h(x,y) = {1 \over 2} {\rm sgn}(y-x)$
and $F(x) = {1 \over 2}$
we have
\begin{equation}\label{1.3a}
{\rm Pf}\,  X = 2^{-(m+1)}\prod_{1 \le j < k \le
2m+1} {\rm sgn} (x_k - x_j),
\end{equation}
where Pf denotes the Pfaffian. Also, it is a simple corollary of the Vandermonde determinant identity that
\begin{equation}\label{1.3b}
\det [R_{k-1}(x_j)]_{j,k=1,\dots,2m+1} \;\propto \prod_{1 \le j < k \le 2m+1} (x_k - x_j).
\end{equation}
Combining (\ref{1.3a}) and (\ref{1.3b}) shows
\begin{equation}\label{1.3c}
\prod_{1 \le j < k \le
2m+1} |x_k - x_j| \;\propto\; \det [R_{k-1}(x_j)]_{j,k=1,\dots,2m+1}\; {\rm Pf} \, X.
\end{equation}

The significance of the decomposition (\ref{1.3c}) for present purposes is that it implies a Pfaffian formula for the generating function
$E_{2m+1,1}$. Specifically, substituting  the definition (\ref{eqn:goedensity}) of the joint density $p_1(x_1,\ldots,x_{2m+1})$ in (\ref{E}) with $\beta  = 1$, then substituting (\ref{1.3c}) we have
\begin{multline}\label{1.3d}
E_{2m+1,1}((-s,s);\xi;w_1) \;\propto\;
\int_{-\omega}^\omega dx_1  \cdots
\int_{-\omega}^\omega dx_{2m+1} \,  \\*[1mm]
\times \prod_{l=1}^{2m+1}(1 - \xi \chi_{x_l \in (-s,s)}) w_1(x_l)
\det [R_{k-1}(x_j)]_{j,k=1,\dots,2m+1} \;{\rm Pf} \, X \\*[2mm]
\;\propto\;
{\rm Pf} \,  \left ( \begin{array}{cc}
[a_{j,k}]_{j,k=1,\dots,2m+1} & [b_j]_{j=1,\dots,2m+1} \\*[2mm]
- [b_k]_{k=1,\dots,2m+1} & 0 \end{array} \right ),
\end{multline}
where $a_{j,k}$, $b_j$ are given by (\ref{0.7}), with the final line being a well known identity in random matrix theory
\cite{dBr55}, \cite[Eq.~(6.84)]{Fo10}.

Finally, to obtain from this the determinant form (\ref{Pr}), note that since $(1 - \xi \chi_{x \in (-s,s)}) w_1(x)$ is even in $x$, and $R_j(x)$ is even (odd) for $j$ even
(odd), we have that $a_{j,k}=0$ when $j,k$ have the same parity, and $b_j=0$ for $j$ even. Thus the nonzero entries in the Pfaffian (\ref{1.3d}) form
a checkerboard pattern. Taking into consideration that $a_{j,k}$ is antisymmetric in the indices $j,k$, rearranging the rows reduces the RHS of (\ref{1.3d}) to
$$
{\rm Pf} \left ( \begin{array}{cc}  0_{m+1} &  Y \\*[2mm]
- Y^T &  0_{m+1}  \end{array} \right ) ,
$$
and this  in turn is equal to $\det Y$. 
\end{proof}

\medskip
At this stage the polynomials $\{R_j(x)\}$, apart from their degree and parity are arbitrary --- a judicious choice takes us closer
to establishing (\ref{B1}). For this, for a given $w_2(x)$ in Table \ref{tab:admissible1}, introduce the family of orthonormal polynomials
$\{ p_j(x) \}_{j=0,\dots,n}$ such that
\begin{equation}\label{OP}
\int_{-\omega}^\omega w_2(x) p_j(x) p_k(x) \, dx = \delta_{jk}.
\end{equation}
In terms of these polynomials, and the pairs of weights as implied by Table \ref{tab:admissible1}, choose
\begin{equation}\label{R}
R_0(x) = 1, \quad R_{2j-1}(x) = p_{2j-1}(x), \quad R_{2j}(x) = - {1 \over w_1(x)} {d \over dx} \left( {w_2(x) \over w_1(x)} p_{2j-1}(x)\right),
\end{equation}
where $j=1,2,\ldots$.
The latter expression is even and a polynomial of degree $j$ since
\[
{1 \over w_1(x)} {d \over dx} {w_2(x) \over w_1(x)} = - {x \over \alpha_1}, \qquad {w_2(x) \over w_1(x)^2} = \phi(x) = \text{even polynomial of degree $2$}, 
\]
following from Eqs. \eqref{eqn:ode2} and \eqref{eqn:w2}, which are constitutive for Table \ref{tab:admissible1}.
We then have, for $j,k=1,2\dots$,
\begin{equation}\label{2.18}
b_1  = {1 \over 2} \int_{-\omega}^\omega w_1(x) (1 - \xi \chi_{x \in (-s,s)}) \, dx, \qquad
b_{2j+1}   =  \xi {w_2(s) \over w_1(s)} p_{2j-1}(s), 
\end{equation}
and
\[
a_{1,2k} = {1 \over 2} \int_{-\omega}^\omega  dx \, w_1(x)  (1 - \xi \chi_{x \in (-s,s)}) 
 \,  \int_{-\omega}^\omega dy \, w_1(y)  (1 - \xi \chi_{y \in (-s,s)}) {\rm sgn} \, (y-x) p_{2k-1}(y),
\]
as well as
\begin{multline}\label{2.18a}
a_{2j+1,2k} = 2 \xi {w_2(s) \over w_1(s)} p_{2j-1}(s) \int_s^\omega w_1(y) p_{2k-1}(y) \, dy \\*[2mm]
- \int_{-\omega}^\omega w_2(y) (1 - \xi \chi_{y \in (-s,s)})^2 p_{2j-1}(y) p_{2k-1}(y) \, dy.  
\end{multline}

One immediate consequence of this choice is that it allows for a simple determination of the proportionality, $1/\theta$ say, in
(\ref{Pr}). Thus with $\xi = 0$ we see that $b_{2j+1} = 0$ and $a_{2j+1,2k} = - \delta_{j,k}$, where to obtain the latter use has
been made of (\ref{OP}). Consequently, cf. \eqref{eq:theta},
\begin{equation}\label{N}
\theta = b_1 |_{\xi = 0} = {1 \over 2} \int_{-\omega}^\omega w_1(x) \, dx.
\end{equation}

Let $C \in O(m)$ be a real \emph{orthogonal} matrix, and define a set $\{q_{2j-1}(x) \}_{j=1,\ldots,m}$ of polynomials by
\begin{equation}\label{q}
\begin{pmatrix}
q_1(x)\\*[1mm]
q_3(x)\\*[1mm]
\vdots\\*[1mm]
q_{2m-1}(x)
\end{pmatrix} = C \begin{pmatrix}
p_1(x)\\*[1mm]
p_3(x)\\*[1mm]
\vdots\\*[1mm]
p_{2m-1}(x)
\end{pmatrix} .
\end{equation}
If $\tilde{Y}$ is defined as for $Y$ but with each occurrence of $p_{2j-1}(x)$ replaced by $q_{2j-1}(x)$,
we get
\begin{equation}\label{Yt}
\begin{pmatrix} 1 &  \\
 & C \end{pmatrix} Y  \begin{pmatrix} C^T & \\
 & 1 \end{pmatrix} = \tilde{Y},\qquad \det \tilde Y = \det Y,
 \end{equation}
 where the latter follows from $|\!\det C| = 1$. This allows us to make the same replacement in (\ref{R}) and thus in (\ref{2.18}) and (\ref{2.18a}) without effecting 
 the representation (\ref{Pr}) of the generating function. That this freedom leads to simplifications can be seen
 from the fact that $\{q_{2j-1}(x) \}$ remains an orthonormal set with respect to the inner product implied by (\ref{OP}), that is,
 \begin{equation}\label{OP1}
\int_{-\omega}^\omega w_2(x) q_{2j-1}(x) q_{2k-1}(x) \, dx = \delta_{jk},
\end{equation}
 but can also be
 chosen to have an additional orthogonality as in the following lemma.
 
 \begin{lemma}
 Define the projection kernel
 \begin{equation}\label{Km}
 K(x,y) = (w_2(x) w_2(y))^{1/2} \sum_{k=1}^m p_{2k-1}(x) p_{2k-1}(y)
 \end{equation}
 together with the associated integral operator
 \begin{equation}\label{Km1}
 K f(x) = \int_{-s}^s K(x,y) f(y) \, dy\qquad (0<s<\omega). 
  \end{equation}
  This integral operator has eigenfunctions $\{ q_{2j-1}(x) \}_{j=1,\dots,m}$ with the structure (\ref{q})
  for some real orthogonal matrix $C$, and furthermore
   \begin{equation}\label{Km2}
\int_{-s}^s  w_2(x) {q}_{2j-1}(x) {q}_{2k-1}(x) \, dx = \nu_{2j-1}(s) \delta_{jk},
\end{equation}
where $0 < \nu_{2j-1}(s)  < 1$ are the eigenvalues of $K$.       
\end{lemma}

This functional analytic result is essentially due to Gaudin \cite{Ga61}; see also \cite[p.~410]{Fo10}.
The determinant of $\tilde Y$ can be simplified by applying 
the  elementary column operations of replacing column $k$
for $k=1,\dots, n$ by column $k$ minus $2 \int_s^\omega w_1(x)  {q}_{2k-1}(x) \, dx$ times column $n+1$.
It is immediate that the entries in rows $2,\dots,n+1$ and columns $1,\dots, n$  are then given by
\begin{multline}\label{at}
\tilde{a}_{2j+1,2k}  = 
 - \int_{-\omega}^\omega w_2(y) (1 - \xi \chi_{y \in (-s,s)})^2 q_{2j-1}(y) q_{2k-1}(y) \, dy\\*[2mm]
 = - \int_{-\omega}^\omega w_2(y)  q_{2j-1}(y) q_{2k-1}(y) \, dy + (2\xi-\xi^2) \int_{-s}^s w_2(y)  q_{2j-1}(y) q_{2k-1}(y) \, dy\\*[2mm]
 = - \delta_{jk} + (2\xi-\xi^2) \nu_{2j-1}\delta_{jk} =  - \delta_{jk} + (1-(\xi-1)^2) \nu_{2j-1}\delta_{jk},
\end{multline}
where we have used \eqref{OP} and \eqref{Km2} to obtain the last line.
The entries in row 1, column $1,\dots,n$, after first simplifying the expression for $a_{1,2k}$ in (\ref{2.18})
by noting that the integral over $y$ can be rewritten according to
\begin{multline*}
{1 \over 2}
 \int_{-\omega}^\omega dy \, w_1(y)  (1 - \xi \chi_{y \in (-s,s)}) {\rm sgn} \, (y-x) p_{2k-1}(y) \\
 = \xi \chi_{x \in (-s,s)} \int_s^\omega w_1(t) q_{2k-1}(t) \, dt +
 (1 - \xi \chi_{x \in (-s,s)}) \int_x^\omega w_1(t) q_{2k-1}(t) \, dt,
 \end{multline*}
 now read
 \begin{multline}\label{at2}
 \tilde{a}_{1,2k}  =  \xi (1 - \xi) \int_s^\omega w_1(t) q_{2k-1}(t) \, dt \; \int_{-s}^s w_1(x) \, dx \;+\;
\int_{-\omega}^\omega w_1(x) (1 - \xi \chi_{x \in (-s,s)})^2   \, dx \\*[2mm]
\times   \int_x^\omega w_1(t) q_{2k-1}(t) \, dt \; - \;
    \int_s^\omega w_1(t) q_{2k-1}(t) \, dt \; \int_{-\omega}^\omega w_1(x) (1 - \xi \chi_{x \in (-s,s)})   \, dx\\*[2mm]
 = \int_{-\omega}^\omega w_1(x) \,dx \int_x^\omega w_1(t)q_{2k-1}(t)\,dt - (1-(\xi-1)^2) \int_{-s}^s w_1(x)\,dx \int_x^s w_1(t)q_{2k-1}(t)\,dt. 
  \end{multline}
  The entries in the final column are unchanged by this process, and thus still have entries $b_{2j-1}$ as specified in
  (\ref{2.18}), with $p_{2j-1}(x)$ replaced by $q_{2j-1}(x)$.
  
To summarize, we have shown with Eqs. \eqref{Pr}, \eqref{N}, \eqref{Yt}, \eqref{at} and \eqref{at2} that
\begin{multline*}
E_{2m+1,1}((-s,s);\xi;w_1) = \frac{1}{\theta} \det \tilde Y =
\begin{vmatrix}
c_1^T + c_2^T (1-(\xi-1)^2) & 1 + \xi \gamma \\*[2mm]
-I + (1-(\xi-1)^2) D & \xi c_3
\end{vmatrix}\\*[4mm]
= \det(I-(1-(\xi-1)^2 D) \;+\; \xi \begin{vmatrix}
c_1^T + c_2^T (1-(\xi-1)^2) & \gamma \\*[2mm]
-I + (1-(\xi-1)^2) D & c_3
\end{vmatrix}
\end{multline*}
with $D = \diag(\nu_1(s),\nu_3(s),\ldots,\nu_{2m-1}(s))$, $\gamma$ a scalar and $c_1, c_2, c_3$ some column vectors with $m$ entries that depend on $s$ but not on $\xi$. The
structure of the last formula is
\begin{subequations}\label{eq:structure}
\begin{equation}
E_{2m+1,1}((-s,s);\xi;w_1) = E(1-(\xi-1)^2) + \xi F(1-(\xi-1)^2),
\end{equation}
where $F(\xi)$ is a polynomial and
\begin{equation}
E(\xi) = \prod_{j=1}^m (1- \xi \nu_{2j-1}(s)).
\end{equation}
\end{subequations}
Now, (\ref{eq:structure}) is immediately amenable to the following simple lemma, which follows from direct computation for the
monomial basis $\{(1-\xi)^j\}_{j=0,1,2,\ldots}$.

\begin{lemma} Let $G(\xi)$ be a polynomial. Then, for $k=0,1,2,\ldots$,
\begin{align*}
\bigg ( {1 \over (2k)!} {\partial^{2k} \over \partial \xi^{2k} }-
{1 \over (2k+1)!} {\partial^{2k+1} \over \partial \xi^{2k+1} } \bigg )
G(1-(\xi-1)^2) \Big |_{\xi = 1} &=  {(-1)^k \over k!} {\partial^{k} \over \partial \xi^{k} }  G(\xi) \Big |_{\xi = 1},\\*[4mm]
\bigg ( {1 \over (2k)!} {\partial^{2k} \over \partial \xi^{2k} }-
{1 \over (2k+1)!} {\partial^{2k+1} \over \partial \xi^{2k+1} } \bigg )
\xi G(1-(\xi-1)^2) \Big |_{\xi = 1} &= 0.
\end{align*}
\end{lemma}

Application of this lemma to (\ref{eq:structure}) gives
\[
\bigg ( {1 \over (2k)!} {\partial^{2k} \over \partial \xi^{2k} }-
{1 \over (2k+1)!} {\partial^{2k+1} \over \partial \xi^{2k+1} } \bigg ) E_{2m+1,1}((-s,s);\xi;w_1) = {(-1)^k \over k!} {\partial^{k} \over \partial \xi^{k} } E(\xi),
\]
which finally proves (\ref{B1}) by the well known and readily established fact (see, e.g.,~\cite[Exercises 9.6 Q.3]{Fo10}) that
$$
 E(\xi) = E_{m,2}((0,s^2); \xi;x^{1/2} w_2(x^{1/2}) \chi_{x > 0}).
 $$
 
 \section{Circular ensembles}\label{section:circular}
 It was remarked in the paragraph including (\ref{UO}) that applying a stereographic projection to the eigenvalues in the
 appropriate Cauchy case of (\ref{1.17}) gives (\ref{CO}). This transformation induces a natural definition of the decimated ensembles ${\rm even} \, |{\rm COE}_n|$ and ${\rm odd} \, |{\rm COE}_n|$. Now, the analogue of Theorem~\ref{thm:main1} allows us to characterize not only the ensemble $\even\,|{\rm COE}_n|$ but also $\odd\,|{\rm COE}_n|$: 
 
\begin{thm}\label{thm:CE} Let $\mu$ be defined as in \eqref{eqn:dim} 
 and, with ${\rm sgn} \, (x) = +$ for $x>0$ and ${\rm sgn} \, (x) = -$ for $x<0$, define $\nu = {\rm sgn} \, (1/2 - \mu)$.
Then, the circular ensembles satisfy the inter-relations
\begin{align}
{\rm even} \, |{\rm COE}_n|  &\,\overset{\rm d}{=}\,  O^\nu(n+1),\label{D2}\\
{\rm odd } \, |{\rm COE}_n|  &\,\overset{\rm d}{=}\,  O^{-\nu}(n+1),\label{D3}\\
|{\rm CUE}_n| &\,\overset{\rm d}{=}\,  {\rm even} \, |{\rm COE}_n|   \, \cup \,  {\rm odd } \, |{\rm COE}_n|,\label{CO1}
\end{align}
where, in the last equation, both ensembles on the right are to be chosen independently.
\end{thm}

\begin{rem}
The last inter-relation should be contrasted with the trivial relation
\[
|{\rm COE}_n| \,\overset{\rm d}{=}\,   {\rm even} \, |{\rm COE}_n|   \, \cup \,  {\rm odd } \, |{\rm COE}_n|
\]
when both occurrences of  ${\rm COE}_n$ on the right would represent one and the \emph{same} ensemble instead of being independent.
\end{rem}
\begin{proof} 
The application of Theorem~\ref{thm:main1} to the Cauchy ensembles with weight \eqref{W1} 
and a subsequent transformation to the circular ensembles by a stereographic projection of the eigenvalues
transforms, by recalling \eqref{UO}, the inter-relation (\ref{eqn:even_aGUE}) into the first assertion \eqref{D2}.

Next, we repeat these steps with the Cauchy weight 
\begin{equation}\label{W1Cauchy-with-a}
w_1(x) = \frac{1}{(1+x^2)^{(n-1+a)/2+1}}\qquad (a>-1),
\end{equation}
which transforms by the stereographic projection into the circular Jacobi ensemble with parameter $a$ \cite[\S3.9]{Fo10}.
Though the resulting PDF becomes singular in the limit $a \to -1^+$, we know from working in the theory of the Selberg integral
  (see e.g.~\cite[Prop. 4.1.3]{Fo10}) that the limit effectively reduces the number of eigenvalues from $n$ to $n-1$, by the mechanism
  of freezing one eigenvalue, taken to be at $\theta = \pi$. This decouples but
  otherwise leaves the joint distribution of the remaining eigenvalues unchanged. Noting that the freezing of an eigenvalue at $\theta = \pi$ also has the consequence of replacing the even operation by the odd operation,
  and after applying analogous reasoning on the RHS of (\ref{eqn:even_aGUE}), we deduce the second assertion \eqref{D3}.

Finally, by recalling (\ref{CO}), the last assertion \eqref{CO1} follows from (\ref{D2}) and (\ref{D3}). 
Alternatively, \eqref{CO1} could also have been deduced from (\ref{eqn:super})  by the choice of the appropriate Cauchy weight, and appropriate interpretation of
the weight in the second term as just discussed.
\end{proof}

\begin{rem}
Interestingly, the pathway to (\ref{D3}) via the limit $a\to-1^+$ in (\ref{eqn:even_aGUE}) with weight (\ref{W1Cauchy-with-a}) can also be followed in the appropriate Laguerre and Jacobi cases of (\ref{24c}) to
deduce (\ref{24a}).
\end{rem}

Analogous to the deduction of Theorem \ref{thm:gap} from  Theorem~\ref{thm:main1}, as a corollary of Theorem~\ref{thm:CE}, we 
get:

\begin{thm}\label{thm:D4} With $\mu$ as in \eqref{eqn:dim} and $\nu = \sign(1/2-\mu)$, we have the gap probability inter-relations\\*[-3mm]
  \begin{align*}
  E_{n,1}(2k-1+\mu;(-\theta,\theta);{\rm COE}_n) + E_{n,1}(2k + \mu ;(-\theta,\theta);{\rm COE}_n)  &=
  E_{m,2}(k;(0,\theta); O^{+\nu}(n+1)),\\*[2mm]
  E_{n,1}(2k-\mu;(-\theta,\theta);{\rm COE}_n) + E_{n,1}(2k + 1 -\mu ;(-\theta,\theta);{\rm COE}_n)  &=
  E_{\hat m,2}(k;(0,\theta); O^{-\nu}(n+1)).
  \end{align*}
 \end{thm}

In the case $n$ even these inter-relations has previously been noted in \cite[Eq.~(3.25)]{EvenSymm}, where it is remarked that it allows the gap probabilities of ${\rm COE}_n$ to be expressed as simple linear combinations of the gap probabilities of $O^{\pm}(n+1)$.
One advantage of such expressions is that the ensembles $O^\pm(n+1)$ are
determinantal point processes (see e.g.~\cite[Ch.~5]{Fo10}), allowing the corresponding gap probabilities to be expressed as Fredholm determinants, which enjoy exponentially
fast numerical approximation, and thus allowing for their efficient high precision computation \cite{Bo10}.
Another advantage is that the gap probabilities for determinantal point processes can be shown to obey a local
limit theorem in an appropriate asymptotic regime. The inter-relations then allow for the deduction of such
asymptotic behaviour for the sum of neighbouring gap probabilities in COE${}_n$, for which no direct
methods are known \cite{FL14}.

The gap probability inter-relation implied by (\ref{CO1}) is exactly (\ref{8.31p}), even though the matrix ensemble inter-relation (\ref{21})  used in its previous derivation
is distinct from (\ref{CO1}). This is a concrete example of the general fact that the family of gap-probability inter-relations specified by
 Theorem \ref{thm:gap} or by Theorem~\ref{thm:D4} do not contain enough information to determine a particular matrix ensemble inter-relation, even though they
 are suggestive.

\section*{Acknowledgements}
The work of FB was supported by the DFG-Collaborative Research Center, TRR
109, ``Discretization in Geometry and Dynamics.''
The work of PJF was supported by the Australian Research Council through the grant DP140102613.


\providecommand{\bysame}{\leavevmode\hbox to3em{\hrulefill}\thinspace}
\providecommand{\MR}{\relax\ifhmode\unskip\space\fi MR }
\providecommand{\MRhref}[2]{%
  \href{http://www.ams.org/mathscinet-getitem?mr=#1}{#2}
}
\providecommand{\href}[2]{#2}


\end{document}